\documentclass[a4paper,12pt]{amsart}

\usepackage{amsfonts}
\usepackage{amsmath}
\usepackage{amssymb}
\usepackage{mathrsfs}
\usepackage[colorlinks]{hyperref}

\usepackage{graphicx}

\usepackage[usenames]{color}

\newtheorem{thm}{Theorem}[section]

\newtheorem{lem}[thm]{Lemma}
\newtheorem{prop}[thm]{Proposition}

\numberwithin{equation}{section}

\theoremstyle{definition}
\newtheorem{definition}[thm]{Definition}

\begin{document}

\title[Prabhakar fractional $q$-integral and $q$-differential operators]
{The Prabhakar fractional $q$-integral and $q$-differential operators,  and their properties}

\author[S. Shaimardan]{Serikbol Shaimardan}
\address{
  Serikbol Shaimardan:
  \endgraf
  L. N. Gumilyov Eurasian National University
  \endgraf
  5 Munaytpasov str., Astana, 010008
  \endgraf
  Kazakhstan
  \endgraf
  and
  \endgraf
Department of Mathematics: Analysis, Logic and Discrete Mathematics
  \endgraf
Ghent University, Krijgslaan 281, Building S8
B 9000 Ghent
\endgraf
Belgium
  \endgraf
  {\it E-mail address} {\rm shaimardan.serik@gmail.com}}

\author[E. Karimov]{Erkinjon Karimov}
\address{
  Erkinjon Karimov:
  \endgraf
  Fergana State University
  \endgraf
 19 Murabbiylar str., Fergana, 140100
  \endgraf
 Uzbekistan
 \endgraf
   and
  \endgraf
  V.I.Romanovskiy Institute of Mathematics
  \endgraf
 9 Universitet str., Tashkent, 100174
 \endgraf
 Uzbekistan
  \endgraf
 {\it E-mail address} {\rm erkinjon@gmail.com}
  }
 \author[M. Ruzhansky ]{Michael Ruzhansky }
\address{
Michael Ruzhansky :
  \endgraf
Department of Mathematics: Analysis, Logic and Discrete Mathematics
  \endgraf
Ghent University, Krijgslaan 281, Building S8
B 9000 Ghent
\endgraf
Belgium
\endgraf
and
\endgraf
School of Mathematical Sciences
Queen Mary University of London
\endgraf
London, United Kingdom
  \endgraf
  {\it E-mail address} {\rm Michael.Ruzhansky@UGent.be}
  }

\author[A. Mamanazarov]{Azizbek Mamanazarov}
\address{
 Azizbek Mamanazarov:
  \endgraf
Fergana State University
  \endgraf
 19 Murabbiylar str., Fergana, 150100
  \endgraf
 Uzbekistan
 \endgraf
 \endgraf
  {\it E-mail address} {\rm mamanazarovaz1992@gmail.com}
  }

\date{}

\begin{abstract}
In this paper, we have introduced the Prabhakar fractional $q$-integral and $q$-differential operators. We first study the semi-group property of the Prabhakar fractional $q$-integral  operator, which allowed us
 to introduce the corresponding $q$-differential  operator. Formulas for compositions of $q$-integral and $q$-differential
 operators are also presented. We show the boundedness of the Prabhakar fractional $q$-integral  operator in the class of $q$-integrable functions.

\end{abstract}



\maketitle

\section{Introduction}

Fractional calculus is the area of mathematical analysis that deals with the study and application
of integrals and derivatives of arbitrary order. In recent decades, fractional calculus has become
of increasing significance due to its applications in many fields of science and engineering. For instance, it has many applications in viscoelasticity, signals processing, electromagnetics, fluid mechanics, and optics. For more information on this research we refer the readers to \cite{Meral}, \cite{Mainardi}, \cite{Machado}, \cite{Kilbas}, \cite{Hilfer}, \cite{Lazo}
and the references therein.

An interesting and distinctive feature of the fractional calculus is that it is possible to present different definitions of fractional order integrals and derivatives; furthermore, many instances of those definitions are being applied and discussed to analyze specific processes \cite{Teodor}.
As an example, we can take the Riemann-Liouville and Caputo fractional order integral-differential operators which have been used widely to describe mathematical models of many natural phenomena (see the references \cite{Uchaikin}, \cite{Quresh}, \cite{Machado}). Recently, researchers have focused on the generalizations of fractional order operators. It can be explained, on the one hand, by the fact that in mathematical modelling of
some real-life processes we get such type of generalization of the fractional operators. On the other hand, it is an inner need of the theory of Fractional Calculus. We wish to focus on Prabhakar fractional $q$-differential and differential operators among these operators.

First, we would like to give some brief information about classical Prabhakar fractional calculus.
The theory of Prabhakar fractional calculus \cite{Garra} has been more intensively studied in recent years
and as a result, certain differential equations involving Prabhakar operators became an intensive target, which is interesting
both for their pure mathematical properties \cite{Restrepo}-\cite{Eshagi} and for their real-world applications
in topics such as viscoelasticity, anomalous dielectrics, and options pricing \cite{Colombaro}, \cite{Garappa}, \cite{Tomovski}.

The origin of the $q$-difference calculus can be traced back to the works \cite{Jakson} and \cite{Karmichel}.
Recently, the fractional $q$-difference calculus has been proposed by W. Alsalam \cite{Alsalam} and R.P.
Agarwal \cite{Agarval}. Nowadays new developments in the theory of fractional $q$-difference calculus have
been addressed extensively by several researchers (see
\cite{Raj}, \cite{Serikbologa} and the references therein ).

In the present work, we aim to introduce and study some properties of Prabhakar fractional $q$-integral and differential operators.

\section{Preliminaries}
First, we recall some elements of the $q$-calculus for the sequel. For more information we note the works
\cite{Kac}, \cite{Zaynab} and the references therein. From now on, we assume that $0<q<1$ and
$0\le a < b < \infty.$

Let $\alpha \in \mathbb{R}$. A $q$-real number $[\alpha]_{q} $ is defined by
\begin{eqnarray*}
  [\alpha]_{q}=
  \frac{1-{q}^{\alpha}}{1-q}.
\end{eqnarray*}

And also, the  $q$-shifted factorial is defined
by  \begin{eqnarray*}
{{\left( a;q \right)}_n}=\left\{ \begin{gathered}
1, \quad \quad \quad \quad n=0; \\
\left( 1-a \right)\left( 1-aq \right) ... \left( 1-a{q^{n-1}} \right), \quad n\in \mathbb{N}. \\
\end{gathered} \right.
\end{eqnarray*}

 The $q$-analogue of the factorial is
\begin{eqnarray}\label{2.1}
    [n]_q!=[1]_q[2]_q[3]_q...[n]_q=\frac{(q;q)_n}{(1-q)^n}, n\in \mathbb{N}, \quad [0]_q!=1.
\end{eqnarray}

For $q$-binomial coefficients we have the following formula
  \begin{eqnarray}\label{2.3}
\left[ \begin{gathered}
  n  \\
  k \\
\end{gathered} \right]_q=\frac{(1-q^n)(1-q^{n-1})...(1-q)^{n-k+1}}{(q;q)_k}=\frac{[n]_q!}{[n-k]_q![k]_q!}.
\end{eqnarray}

Also, the $q$-analogue of the power $(a-b)_{q}^k$ is defined by
\begin{eqnarray*}
    (a-b)_q^0=1, \quad (a-b)_q^k=\prod\limits_{i=0}^{k-1}{\left( a-bq^i \right)}, \quad k\in \mathbb{N}.
\end{eqnarray*}

There is the following relationship between them:
\begin{eqnarray*}
 (a-b)^0_q=1; \quad  (a-b)^k_q=a^k(b/a;q)_k, \quad a\ne 0, \quad k \in \mathbb{N},
\end{eqnarray*}
as well as
\begin{eqnarray*}
(a-b)^{\alpha}_q=
a^{\alpha} \frac{(b/a;q)_{\infty}}{(q^{\alpha}b/a;q)_{\infty}},
\quad   {\left( a;q \right)}_{\alpha }=
\frac{{\left( a;q \right)}_{\infty }}{\left( a{{q}^{\alpha }};q \right)}_{\infty }, \quad
(a;q)_{\infty}=\prod\limits_{i=0}^{\infty}\left( 1-a{{q}^{i}} \right).
\end{eqnarray*}

For $x>0$ the $q$-analogue of the gamma function is defined by
\begin{eqnarray}\label{2.4}
  {\Gamma }_q\left( x \right)=
  \frac{{\left( q;q \right)}_\infty }{{\left( {{q}^{x}};q \right)}_\infty }{\left( 1-q \right)}^{1-x}.
\end{eqnarray}

It has the following property
\begin{eqnarray}\label{2.5}
    \Gamma_q(x+1)=[x]_q\Gamma_q(x).
\end{eqnarray}

The (Jackson) $q$-derivative of a function $f(x)$ is defined by
\begin{eqnarray*}
  \left( D_qf \right)\left( x \right)=
  \frac{f\left( x \right)-f\left( qx \right)}{x\left( 1-q \right)}, \quad \left( x\ne 0 \right)\quad
\end{eqnarray*}
and $q$-derivatives $D_q^nf$ of
higher order are defined inductively as follows:
\begin{eqnarray*}
  D_{q}^{0}f=f,
  \quad  D_{q}^{n}f={D_q}\left( D_{q}^{n-1}f \right)\quad (n=1,2,3,...) .
\end{eqnarray*}

Moreover,
\begin{eqnarray}\label{2.6}
    D_q[(x-b)^{\alpha}_q]=[\alpha]_q(x-b)^{\alpha-1}_q,
\end{eqnarray}
\begin{eqnarray}\label{2.7}
    D_q[(a-x)^{\alpha}_q]=-[\alpha]_q(a-qx)^{\alpha-1}_q.
\end{eqnarray}

The $q$-integral (Jackson integral) is defined by
\begin{eqnarray*}
  \left( I_{q,0+}f \right)\left( x \right)=
  \int\limits_0^x f\left( t \right)d_qt=
  x\left( 1-q \right)\sum\limits_{k=0}^{\infty }{f\left( xq^k \right)q^k,}
\end{eqnarray*}
and
\begin{eqnarray*}
  \left( I_{q,a+}f \right)\left( x \right)=\int\limits_a^xf\left( t \right)d_qt=\int\limits_0^xf\left( t \right)d_qt-\int\limits_0^af\left( t \right)d_qt.
\end{eqnarray*}

For the $n$-th order integral operator $I_{q,a}^{n}$ we have
\begin{eqnarray*}
(I_{q,a+}^0f)(x)=f(x),\,\,(I_{q,a+}^nf)(x)=I_{q,a+}\left( I_{q,a+}^{n-1}f \right)(x)\quad \,\,\quad
\left( n=0,1,2,\cdots \right).
\end{eqnarray*}

And also between $q$-integral and $q$-derivative operators, we have the following relations:
\begin{eqnarray*}
  \left( D_qI_{q,a+}f \right)\left( x \right)=f\left( x \right),\,\,\,\left( I_{q,a+}D_qf \right)\left( x \right)=f\left( x \right)-f\left( a \right).
\end{eqnarray*}

For   $\alpha, \beta>0$  and $z\in \mathbb{R}$, a  $q$-analogue of the Mittag–Leffler function is defined as follows (\cite{Zaynab}):
\begin{eqnarray}\label{2.8}
  e_{\alpha ,\beta }\left( z;q \right)=\sum\limits_{n=0}^{\infty }{\frac{z^n}{\Gamma_q\left( \alpha n +\beta  \right)}, \quad  \quad \left( \left| z{{\left( 1-q \right)}^\alpha } \right|<1 \right)}.
\end{eqnarray}

\begin{definition}\label{def2.1}
(\cite{Nadeem}) Let $\alpha ,\beta ,\gamma,z \in \mathbb{R}$ such that $\alpha,\beta>0$. Then the $q$-Prabhakar function $e_{\alpha ,\beta }^{\gamma }\left(z;q \right)$
is defined by
\begin{eqnarray}\label{2.9}
  e_{\alpha ,\beta }^{\gamma }\left( z;q \right)=\sum\limits_{n=0}^{\infty }{\frac{(\gamma)_{n,q}z^n}{\Gamma _q\left( \alpha n+\beta  \right)}}, \quad   \quad |z(1-q)^{\alpha}|<1,
\end{eqnarray}
where
\begin{eqnarray}\label{2.10}
(\gamma)_{n,q}:=\frac{{\left( q^{\gamma };q \right)}_n}{\left( q;q \right)_n}.
\end{eqnarray}
    \end{definition}

\begin{lem} \label{Lemma2.2} ( \cite{Zaynab})
    Let  $\alpha$ and $\beta$ be two complex numbers. Then
    \begin{eqnarray}\label{2.11}        (q^{\alpha+\beta};q)_n=\sum\limits_{k=0}^{n}\left[ \begin{gathered}
  n  \\
  k \\
\end{gathered} \right]_q q^{k\beta}(q^{\alpha};q)_k(q^{\beta};q)_{n-k},\quad \,\,\quad
\left( n=0,1,2,\cdots \right).
    \end{eqnarray}
\end{lem}

\begin{prop}\label{prop2.3}
Let $\gamma,\sigma\in\mathbb{C}$. Then  the following equality is valid
\begin{eqnarray}\label{2.12}
  \sum\limits_{k=0}^{n} \left( \gamma  \right)_{n-k,q}q^{\gamma k}\left( \sigma  \right)_{k,q}= \left( \gamma +\sigma  \right)_{n,q},\quad \,\,\quad
\left( n=0,1,2,\cdots \right).
\end{eqnarray}
\end{prop}

\begin{proof}
    Taking (\ref{2.10}) into account, we rewrite (\ref{2.12}) in the form
    \begin{eqnarray*}
        \sum\limits_{k=0}^{n} q^{k\gamma}\frac{{\left( q^{\gamma };q \right)}_{n-k}}{\left( q;q \right)_{n-k}}\frac{{\left( q^{\sigma};q \right)}_k}{\left( q;q \right)_k}=\frac{{\left( q^{\gamma+\sigma};q \right)}_n}{\left( q;q \right)_n}.
    \end{eqnarray*}

To prove Proposition \ref{prop2.3} it is sufficient to show the validity of the last equality. For this aim, we multiply both sides of the last equality by $(q;q)_n$ and taking into account (\ref{2.1}) and (\ref{2.3}), we obtain
\begin{eqnarray*}
    {\left( q^{\gamma+\sigma};q \right)}_n&=&        \sum\limits_{k=0}^{n} \frac{{\left( q;q \right)}_{n}}{\left( q;q \right)_{n-k}(q;q)_{k}}q^{k\gamma}{{\left( q^{\sigma};q \right)}_k}{\left( q^{\gamma};q \right)_{n-k}}\\&=&
\sum\limits_{k=0}^{n} \frac{\frac{{\left( q;q \right)}_{n}}{(1-q)^n}}{\frac{\left( q;q \right)_{n-k}}{(1-q)^{n-k}}\frac{(q;q)_{k}}{(1-q)^{k}}}q^{k\gamma}{{\left( q^{\sigma};q \right)}_k}{\left( q^{\gamma};q \right)_{n-k}}\\&=&
\sum\limits_{k=0}^{n}\frac{[n]_q!}{[n-k]_q![k]_q!}q^{k\gamma}{{\left( q^{\sigma};q \right)}_k}{\left( q^{\gamma};q \right)_{n-k}}
\\&=&
\sum\limits_{k=0}^{n}\left[ \begin{gathered}
  n  \\
  k \\
\end{gathered} \right]_q q^{k\gamma}(q^{\sigma};q)_k(q^{\gamma};q)_{n-k}.
\end{eqnarray*}

Using the result of Lemma \ref{Lemma2.2} when $\alpha=\sigma$ and $\beta=\gamma$, we get the proof of the Proposition  \ref{prop2.3}.
\end{proof}

Now, we introduce a generalized $q$-Prabhakar function.

\begin{definition}\label{def2.2}
Let $\alpha ,\beta ,\gamma, \omega, \delta, z, s  \in \mathbb{R}$ be such that $\alpha,\beta>0$ and $s<z$. Then the generalized $q$-Prabhakar function $e_{\alpha ,\beta }^{\gamma }$
is defined by
\begin{eqnarray}\label{2.13}
  e_{\alpha ,\beta }^{\gamma }\left[ \omega(z-s)_{q}^{\delta};q \right]:=\sum\limits_{n=0}^{\infty }{\frac{(\gamma)_{n,q}{\omega}^{n}(z-s)_{q}^{\delta n}}{\Gamma _q\left( \alpha n+\beta  \right)}},
\end{eqnarray}
where $|\omega(z-s)_{q}^{\delta}|<(1-q)^{-\alpha}$.
\end{definition}

We note that (\ref{2.13}) can be considered a generalization of some known functions. For instance, if  $\delta=\omega=1, s=0$ then  from (\ref{2.13}) we get  Definition \ref{def2.1} of $q$-Prabhakar function. And also when  $ \gamma=0$ and $\delta=\omega=1$ then  from (\ref{2.13}) we get formula (\ref{2.8}) for the $q$-Mittag-Leffler function.

Now, we give basic concepts of the $q$-fractional calculus.

\begin{definition}\label{def2.5}(\cite{Raj})
    The Riemann-Liouville $q$-fractional integral $I_{q,a+}^{\alpha }$ of order $\alpha > 0$
is defined by
\begin{eqnarray}\label{2.14}
  \left( I_{q,a+}^{\alpha }f \right)\left( x \right)=\frac{1}{\Gamma_q\left( \alpha  \right)}\int\limits_a^x{{\left( x-qt \right)}^{\alpha -1}_q}f\left( t \right)d_qt.
\end{eqnarray}
\end{definition}

\begin{definition}(\cite{Raj})\label{def2.6}
The Riemann-Liouville $q$-fractional differential operator $D_{q,a+}^{\alpha }f$ of order $\alpha > 0$
is defined by
\begin{eqnarray}\label{2.15}
  \left( D_{q,a+}^{\alpha }f \right)\left( x \right)=\left( D_{q,a+}^{\lceil \alpha \rceil}I_{q,a+}^{\lceil \alpha \rceil-\alpha }f \right)\left( x \right),
\end{eqnarray}
where $\lceil \alpha \rceil$ denotes the smallest integer greater or equal to $\alpha$.
\end{definition}

Notice that for $\lambda \in (-1,+\infty)$, we have
\begin{eqnarray}\label{2.16}
    I_{q,a+}^{\alpha }\left( \left( x-a \right)_q^{\lambda } \right)=\frac{\Gamma_q\left( \lambda +1 \right)}{\Gamma _q\left( \alpha +\lambda +1 \right)}\left( x-a \right)_q^{ \alpha +\lambda  }.
\end{eqnarray}

For $1\le p<\infty $ the space  $L_q^p\left[ a,b \right]$  is defined by (\cite{Serikbologa})
\begin{eqnarray*}  L_q^p\left[ a,b \right]=\left\{ f:\left[ a,b \right]\to \mathbb{C}:{{\left( \int\limits_a^b\left| f\left( x \right) \right|^pd_qx \right)}^{1/p}}<\infty  \right\}.
\end{eqnarray*}

\begin{definition}\label{def2.7}(\cite{Zaynab})
 A function $f:[a,b] \to \mathbb{R}$ is called $q$-absolutely continuous if $\exists\varphi \in L_q^1[a,b]$ such that
\begin{eqnarray*} f(x)=f(a)+\int\limits_a^x{\varphi(t)d_qt}
\end{eqnarray*}
for all $x\in [a,b]$.
\end{definition}

The set of all $q$-absolutely continuous functions defined on $[a,b]$ is denoted by  $AC_q[a, b]$. Moreover, $AC_{q}^{n}[a, b]$ ($n\in\mathbb{N}$) is the space of real-valued functions $f(x)$ which have $q$-derivatives up to order $n-1$ on $[a,b]$ such that $D_{q}^{n-1}f \in  AC_q[a, b]$, i.e.
\begin{eqnarray*}
    AC_{q}^{n}[a, b]=\left\{ f: [a,b]\to \mathbb{R}; (D_q^{n-1}f)(x)\in AC_q[a,b]
 \right\}.
\end{eqnarray*}

\begin{lem}\label{lemma2.8} ( \cite{Serikbologa})
a) \quad Let \quad $\alpha >0$, $\beta >0$ and $1\le p<\infty$. Then the $q$-fractional integral has the following semi-group property
\begin{eqnarray*}
  \left( I_{q,a+}^{\alpha }I_{q,a+}^{\beta }f \right)\left( x \right)=\left( I_{q,a+}^{\alpha +\beta }f \right)\left( x \right)
\end{eqnarray*}
for all $x\in \left[ a,b \right]$ and $f\in L_q^p\left[ a,b \right]$.

b) Let $\alpha >\beta >0,\,\,1\le p<\infty $ and $f\in L_q^p\left[ a,b \right]$. Then the following equalities
\begin{eqnarray*}
  \left( D_{q,a+}^{\alpha }I_{q,a+}^{\alpha } \right)\left( x \right)=f\left( x \right), \quad \,\,\,\left( D_{q,a+}^{\beta }I_{q,a+}^{\alpha }f \right)\left( x \right)=\left( I_{q,a+}^{\alpha -\beta }f \right)\left( x \right)
\end{eqnarray*}
hold for all $x\in \left[ a,b \right]$.
    \end{lem}

\begin{lem}\label{lemma} (\cite{Serikbologa})
Let \quad $\alpha>0$  and  $1\leq p<\infty$. Then the $q$-fractional integral operator  $I_{q,a+}^{\alpha}$ is bounded in $L_q^{p}[a,b]$:
\begin{eqnarray}\label{2.17}
    \left\| {I_{q,a+}^{\alpha}}f \right\|_{L_q^p[a,b]}\leq K \left\| {f} \right\|_{L_q^p[a,b]},
    \end{eqnarray}
where $K=\frac{(b-qa)_q^{\alpha}}{\Gamma_q(\alpha+1)}$.
    \end{lem}

\section{ The definitions and the main properties of the Prabhakar fractional $q$- integral and $q$-differential operators}

In this section, using Definition \ref{def2.2} of  the generalized $q$-Prabhakar function  we introduce Prabhakar fractional $q$-integral and $q$-differential operators.

\begin{definition}\label{def3.1}
Let $f\in L_{q}^{1}\left[ a,b \right]$ and  $\alpha ,\beta, \gamma, \omega \in \mathbb{R}$ be such that $ \alpha,\beta>0$. Then the Prabhakar fractional $q$-integral operator is defined by
\begin{eqnarray}\label{3.1}
  \left(_{}^{P}I_{q,a+}^{\alpha ,\beta ,\gamma ,\omega }f\right) \left( x \right) :=
  \int\limits_{a}^{x}{\left( x-qt \right)_{q}^{\beta -1}e_{\alpha ,\beta }^{\gamma }\left[\omega\left( x-q^{\beta}t \right)_{q}^{\alpha };q \right] f(t)d_qt}.
\end{eqnarray}
\end{definition}

From here and for the rest of the paper we denote $\omega^{\prime}=q^{\gamma}\omega$.

\begin{prop}\label{prop3.1}
 Let $\alpha ,\beta, \gamma,$
 $\mu,$ $\sigma$, $\omega \in \mathbb{R}$ be such that $ \alpha,\beta, \mu >0$ and
$x,s \in \mathbb{R}^{+}$, $x>s$.
Then
\begin{eqnarray}\label{eq3.2}
 _{}^{P}I_{q, qs+}^{\alpha, \beta, \gamma ,\omega}
  \left\{g_{\sigma, \omega'}^{\alpha,\mu}(x,s) \right\}=g_{\sigma+\gamma,\omega}^{\alpha, \mu+\beta}(x,s),
  \end{eqnarray}
where
\begin{eqnarray}\label{3.3}
g_{\sigma,\omega'}^{\alpha,\mu}(x,s):=\left( x-qs \right)_{q}^{\mu -1}e_{\alpha ,\mu}^{\sigma }\left[\omega^{\prime}\left( x-q^{\mu}s \right)_{q}^{\alpha };q \right].
\end{eqnarray}
\end{prop}

\begin{proof}
By using (\ref{3.1}), we have
\begin{eqnarray*}
    _{}^{P}I_{q, qs+}^{\alpha, \beta, \gamma ,\omega}
  \left\{g_{\sigma, \omega'}^{\alpha,\mu}(x,s) \right\}= {\int\limits_{qs}^x g_{\gamma,\omega}^{\alpha,\beta}(x,t) g_{\sigma,\omega'}^{\alpha, \mu}(x,s)}d_qt.
\end{eqnarray*}

Taking into account (\ref{2.13}) and using Definition \ref{def2.5}, we have
\begin{eqnarray*}
_{}^{P}I_{q, qs+}^{\alpha, \beta, \gamma ,\omega}
  \left\{g_{\sigma, \omega'}^{\alpha,\mu}(x,s) \right\}=\sum\limits_{n=0}^{\infty }
    {\left( \gamma  \right)}_{n,q}{\omega }^n\sum\limits_{k=0}^{\infty }{\frac{{q^{\gamma k}{\omega}^{k}\left( \sigma  \right)_{k,q}}}{{{\Gamma }_{q}}\left( \alpha k+\mu  \right)}I_{q,qs+}^{\alpha n+\beta }\left( x-qs \right)_{q}^{\alpha k+\mu -1}}.
\end{eqnarray*}
Hence, applying (\ref{2.16}) we obtain
\begin{eqnarray*}
_{}^{P}I_{q, qs+}^{\alpha, \beta, \gamma ,\omega}
  \left\{g_{\sigma, \omega'}^{\alpha,\mu}(x,s) \right\}=\sum\limits_{n=0}^{\infty }{{{\left( \gamma  \right)}_{n,q}}{{\omega }^{n}}}\sum\limits_{k=0}^{\infty }{{\left( \sigma  \right)}_{k,q}}q^{\gamma k}\omega^k \frac{{\left( x-qs \right)}_{q}^{\alpha k+\alpha n+\beta +\mu -1}}{{{\Gamma }_{q}}\left( \alpha n+\alpha k+\beta +\mu  \right)}.
\end{eqnarray*}

Using the Cauchy product formula (\cite{Kanutto}) and then taking into account (\ref{2.12}) and also the expansion  (\ref{2.13})  of the generalized $q$-Prabhakar function,  we derive
\begin{eqnarray*}
_{}^{P}I_{q, qs+}^{\alpha, \beta, \gamma ,\omega}
  \left\{g_{\sigma, \omega'}^{\alpha,\mu}(x,s) \right\}&=& \sum\limits_{n=0}^{\infty }{\frac{{{\left( x-qs \right)}_{q}^{\alpha n+\beta +\mu -1}}{{\omega }^{n}}}{{{\Gamma }_{q}}\left( \alpha n+\beta +\mu  \right)}}\sum\limits_{k=0}^{n}{{{\left( \gamma  \right)}_{n-k,q}}{{q}^{\gamma k}}{{\left( \sigma  \right)}_{k,q}}}\\&=&     \sum\limits_{n=0}^{\infty }{\frac{{{\left( \gamma +\sigma  \right)}_{n,q}}{{\omega }^{n}}}{{{\Gamma }_{q}}\left( \alpha n+\beta +\mu  \right)}}{{\left( x-qs \right)}_{q}^{\alpha n+\beta +\mu -1}}\\&=&
    \left( x-qs \right)_{q}^{\beta +\mu -1}e_{\alpha ,\beta +\mu }^{\gamma +\sigma }\left[ \omega \left( x-{{q}^{\beta +\mu }}s \right)_{q}^{\alpha };q \right]\\&=&g_{\sigma+\gamma,\omega}^{\alpha, \mu+\beta}(x,s),
\end{eqnarray*}
completing the proof.
\end{proof}

\begin{lem}\label{lemma3.3}  Let $f\in L^{p}_{q}[a,b]$ and $\alpha$, $\beta, \gamma, \omega, $ $\mu$ , $\sigma\in \mathbb{R}$ be such that $\alpha, \beta, \mu >0$. Then the following relation
\begin{eqnarray}\label{3.4}
\left(_{}^{P}I_{q,a+}^{\alpha, \beta, \gamma ,\omega} {_{}^{P}I_{q,a+}^{\alpha, \mu, \sigma,\omega'}} f \right)(x)=\left(
 _{}^{P}I_{q,a+}^{\alpha, \beta+\mu, \gamma +\sigma ,\omega} f \right)(x)
\end{eqnarray}
holds for all $x\in[a,b].$

In particular
\begin{eqnarray}\label{3.5}
\left( _{}^{P}I_{q,a+}^{\alpha, \beta, \gamma ,\omega} {_{}^{P}I_{q,a+}^{\alpha, \mu, -\gamma,\omega'}} f\right)(x)
    =\left(I_{q,a+}^{ \beta+\mu}f\right)(x) .
\end{eqnarray}
\end{lem}

\begin{proof}
By Definition  \ref{3.1} of the Prabhakar fractional $q$-integral operator and taking into account notation (\ref{3.3}), we have \begin{eqnarray*}
    \left(_{}^{P}I_{q,a+}^{\alpha, \beta, \gamma ,\omega} {_{}^{P}I_{q,a+}^{\alpha, \mu, \sigma,\omega'}} f \right)(x)=
\int\limits_{a}^{x}g_{\gamma,\omega}^{\alpha,\beta}(x,t)\int\limits_{a}^{t}g_{\sigma,\omega'}^{\alpha,\mu}(x,s)f(s)
    d_qs
    d_qt.
   \end{eqnarray*}

Hence, by changing the order of integration and using Definition \ref{def3.1}, we get
\begin{eqnarray*}
    \left(_{}^{P}I_{q,a+}^{\alpha, \beta, \gamma ,\omega} {_{}^{P}I_{q,a+}^{\alpha, \mu, \sigma,\omega'}} f\right)(x)&=&\int\limits_{a}^{x}f(s)
    d_qs\int\limits_{qs}^{x}g_{\gamma,\omega}^{\alpha,\beta}(x,t)g_{\sigma,\omega'}^{\alpha,\mu}(x,s)
    d_qt\\&=&
    \int\limits_{a}^{x}{_{}^{P}I_{q, qs+}^{\alpha, \beta, \gamma ,\omega}}[g_{\sigma,\omega'}^{\alpha,\mu}(x,s)]
  f(s)d_qs.
\end{eqnarray*}

Applying Proposition \ref{prop3.1} and taking into account Definition \ref{def3.1}, we obtain
\begin{eqnarray*}
    \left(_{}^{P}I_{q,a+}^{\alpha, \beta, \gamma ,\omega}
 {_{}^{P}I_{q,a+}^{\alpha, \mu, \sigma,\omega'}} f\right)(x)=
       \int\limits_{a}^{x} g_{\gamma+\sigma,\omega}^{\alpha, \beta+\mu}(x,s)f(s)d_qs=\left( _{}^{P}I_{q,a+}^{\alpha,\beta+\mu, \gamma +\sigma ,\omega} f\right)(x).
\end{eqnarray*}

By putting $\sigma=-\gamma$ and taking $e^{0}_{\alpha,\beta}(z)=1/{\Gamma_q(\beta)}$ into account from the last one can easily obtain  (\ref{3.5}).

The proof of Lemma \ref{lemma3.3} is complete.
\end{proof}

\begin{lem}\label{lem3.4}
Let $\alpha, \beta, \gamma, \omega \in \mathbb{R}$ be such that $\alpha, \beta>0$ $|\gamma|<1$, $|\omega(b-q^{\beta+1}a)_{q}^{\alpha}|<(1-q)^{\alpha}$ and $1\leq p <\infty$.
Then the Prabhakar fractional $q$-integral operator  $_{}^{P}I_{q,a+}^{\alpha, \beta, \gamma,\omega}$ is bounded in $ L_{q}^{p}[a,b]$:
\begin{eqnarray}\label{eq2.11}
    \left\| {_{}^{P}I_{q,a+}^{\alpha, \beta, \gamma,\omega}f} \right\|_{L_{q}^{p}[a,b]}\le M \left\| {f} \right\|_{L_{q}^{p}[a,b]},
\end{eqnarray}
where
\begin{eqnarray}\label{cons}
    M=(b-qa)_q^{\beta}e_{\alpha, \beta+1}[(b-q^{\beta+1}a)_{q}^{\alpha};q].
\end{eqnarray}
\end{lem}

\begin{proof}
 Taking into account Definition \ref{def3.1} and notation (\ref{3.3}), we have
\begin{eqnarray}\label{3.7}
 \left\| {_{}^{P}I_{q,a+}^{\alpha, \beta, \gamma,\omega}f} \right\|^p_{L_{q}^{p}[a,b]}=  \int\limits_{a}^{b}\left| \int\limits_{a}^{x}{g_{\gamma,\omega}^{\alpha,\beta}(x,t)f(t)d_qt}\right|^{p}d_qx \leq  \int\limits_{a}^{b}J_2(x)d_qx,
        \end{eqnarray}
where
\begin{eqnarray*}
 J_2(x):=\left\{\int\limits_{a}^{x}{g_{\gamma,|\omega|}^{\alpha,\beta}(x,t)|f(t)|d_qt}\right\}^p.
\end{eqnarray*}

For $p>1$,  we  define $p^{\prime}$ from the equality $\frac{1}{p}+\frac{1}{p^{\prime}}=1$. Applying the H\"{o}lder-Rogers inequality to $J_2(x)$,  we get
\begin{eqnarray}
 J_2(x) &\leq&
    \left( \int\limits_{a}^{x}{g_{\gamma,|\omega|}^{\alpha,\beta}(x,t)d_qt}\right)^{{p}/{p^{\prime}}}\nonumber  \int\limits_{a}^{x}{g_{\gamma,|\omega|}^{\alpha,\beta}(x,t)|f(t)|^{p}d_qt}\\
  &=&J^{{p}/{p'}}_{21}(x)\times J_{22}(x),
\end{eqnarray}
where
\begin{eqnarray*}
    J_{21}(x)&:=&\int\limits_{a}^{x}{g_{\gamma,|\omega|}^{\alpha,\beta}(x,t)d_qt},\\  J_{22}(x)&=&\int\limits_{a}^{x}{g_{\gamma,|\omega|}^{\alpha,\beta}(x,t)|f(t)|^{p}d_qt}.
\end{eqnarray*}

Let us consider $J_{21}(x)$. We show that the following inequality is true:
\begin{eqnarray}\label{3.9}
    J_{21}(x)\leq M,
\end{eqnarray}
where $M$ is a constant defined by (\ref{cons}). Indeed, taking into account (\ref{3.3}) and (\ref{2.13}) and using formulas  (\ref{2.5}), (\ref{2.7}), we have  \begin{eqnarray*}
J_{21}(x)&=&     \sum\limits_{n=0}^{\infty}\frac{(\gamma)_{n,q}|\omega|^{n}}{\Gamma_q(\alpha n+\beta+1)}{(x-a)_{q}^{\alpha n+\beta}}.
  \end{eqnarray*}
 Since $|\gamma|<1$, we have $(\gamma)_{n,q}<1$. Taking this into account and using \begin{eqnarray*}
 (b-a)_{q}^{\delta}<(b-aq)_{q}^{\delta} \quad (\delta>0),
 \end{eqnarray*}
 we have
       \begin{eqnarray*}
J_{21} (x)      &\leq&     \sum\limits_{n=0}^{\infty}\frac{|\omega|^{n}}{\Gamma_q(\alpha n+\beta+1)}{(b-a)_{q}^{\alpha n+\beta}}
              \\&\leq&     \sum\limits_{n=0}^{\infty}\frac{|\omega|^{n}}{\Gamma_q(\alpha n+\beta+1)}{(b-aq)_{q}^{\alpha n+\beta}}
              \\&=&(b-aq)^{\beta}     \sum\limits_{n=0}^{\infty}\frac{|\omega|^{n}{(b-a^{\beta+1}q)_{q}^{\alpha n}}}{\Gamma_q(\alpha n+\beta+1)}\\&=&
              (b-qa)_q^{\beta}e_{\alpha, \beta+1}[\omega(b-q^{\beta+1}a)_{q}^{\alpha};q]=M.
         \end{eqnarray*}

We note that  the conditions $\alpha, \beta>0$ and $|\omega(b-q^{\beta+1}a)_{q}^{\alpha}|<(1-q)^{\alpha}$ justify the  convergence of series of the function $e_{\alpha, \beta+1}[\omega(b-q^{\beta+1}a)_{q}^{\alpha};q]$.

Then, by virtue of (\ref{3.9}) , we obtain the following inequality
\begin{eqnarray*}
J_2(x)\leq M^{\frac{p}{p^{\prime}}}J_{22}(x).
\end{eqnarray*}

Taking the last inequality  into account  from  (\ref{3.7}), we get
\begin{eqnarray}\label{3.10}
\left\| {_{}^{P}I_{q,a+}^{\alpha, \beta, \gamma,\omega}f} \right\|^p_{L_{q}^{p}[a,b]} &\leq& M^{\frac{p}{p'}}
         \int\limits_{a}^{b} J_{22}(x) d_qx.
\end{eqnarray}

Substituting the expression of $J_{22}(x)$ into (\ref{3.10}), changing the order of integration, and using (\ref{3.3}),(\ref{2.13}) and (\ref{2.6}), we get
\begin{eqnarray*}
J_q(f)&\leq & M^{\frac{p}{p'}}
         \int\limits_{a}^{b}|f(t)|^{p}d_qt\int\limits_{qt}^{b}g_{\gamma,|\omega|}^{\alpha,\beta}(x,t)d_qx        \\ &=&
M^\frac{p}{p'}
      \int\limits_{a}^{b}|f(t)|^{p}d_qt\sum\limits_{n=0}^{\infty}\frac{(\gamma)_{n,q}|\omega|^{n}}{\Gamma_q(\alpha n+\beta)} \int\limits_{qt}^{b}(x-t)_{q}^{\alpha n+\beta-1}d_qx
        \\ &=&
M^\frac{p}{p'}
 \sum\limits_{n=0}^{\infty}\frac{(\gamma)_{n,q}|\omega|^{n}}{\Gamma_q(\alpha n+\beta+1)}\int\limits_{a}^{b}|f(t)|^{p} (b-qt)_{q}^{\alpha n+\beta}d_qt
          \\ &\leq&
M^\frac{p}{p'}
  \sum\limits_{n=0}^{\infty}\frac{(\gamma)_{n,q}|\omega|^{n}}{\Gamma_q(\alpha n+\beta+1)}(b-qa)_{q}^{\alpha n+\beta}\int\limits_{a}^{b}|f(t)|^{p} d_qt
        \\ &=&
M^{\frac{p}{p'}+1}
       \int\limits_{a}^{b}|f(t)|^{p} d_qt  \\
       &=&M^p\left\| {f} \right\|^p_{L_{q}^{p}[a,b]}.
        \end{eqnarray*}

Lemma \ref{3.4} is proved.
    \end{proof}

 Now, we give the definition of the Prabhakar fractional $q$-differential operator.

\begin{definition}\label{def3.5}
Let $f\in L_q^{1}[a,b] $ , ${_{}^{P}}I_{q,a+}^{\alpha ,n -\beta ,-\gamma ,\omega} f \in AC_{q}^{n}[a,b]$ and     $\alpha ,\beta ,\gamma ,\delta \in \mathbb{R}$  with $ \alpha>0$ and $\beta>0$. Then the Prabhakar fractional $q$-differential operator $_{}^{P}D_{a,a+}^{\alpha,\beta,\gamma, \omega }$ is defined by
\begin{eqnarray} \label{3.11} \left(_{}^{P}D_{q,a+}^{\alpha ,\beta ,\gamma, \omega } f\right)\left( x \right):=\left(D_{q,a+}^{n}{_{}^{P}}I_{q,a+}^{\alpha ,n -\beta ,-\gamma ,\omega} f\right) \left( x \right),
\end{eqnarray}
where $n=\lceil \beta \rceil$.
\end{definition}

\begin{thm}\label{Theorem3.6}

Let $\alpha,$ $\beta$, $\gamma$,  $\omega\in \mathbb{R}$ with $\alpha>0$ and $\beta>0$. Then for any function $f\in L_{q}^{1}[a,b]$ the following equality is valid
\begin{eqnarray}\label{3.12}
\left(_{}^{P}D_{q,a+}^{\alpha ,\beta ,\gamma ,\omega }{_{}^{P}}I_{q,a+}^{\alpha ,\beta ,\gamma ,\omega' }f\right)\left( x \right)=f(x).
\end{eqnarray}
\end{thm}
\begin{proof}
Using Definition \ref{def3.5} and formula (\ref{3.4}) and also Lemma \ref{lemma2.8}, we have
\begin{eqnarray*}    _{}^{P}D_{q,a+}^{\alpha ,\beta ,\gamma,
    \omega } \left( _{}^{P}I_{q,a+}^{\alpha ,\beta ,\gamma ,\omega'}f \right)\left( x \right)&=&D_{q,a+}^{n}\left(
 {_{}^{P}I_{q,a+}^{\alpha ,n-\beta ,-\gamma, \omega}}
 _{}^{P}I_{q,a+}^{\alpha ,\beta ,\gamma,\omega' }f \right)\left( x \right)\\&=&
  D_{q,a+}^{n}\left( I_{q,a+}^{n}f \right)\left( x \right)=f\left( x \right).
\end{eqnarray*}
The proof is complete. \end{proof}

In the classical case, Prabhakar fractional integral operators' semi-group property is commutative, but in the $q$-calculus case this property is non-commutative. To deal with this problem we need to introduce the following operator which affects only one parameter of the Prabhakar fractional $q$-operators.

We introduce the  operator $\Lambda_{q}^{n\gamma,\omega}$ defined by setting
\begin{eqnarray*}
    \Lambda_{q}^{n\gamma,\omega}\omega:=q^{n\gamma}\omega, \quad n \in \mathbb{N}.
\end{eqnarray*}
 For instance,  $\Lambda_{q}^{n\gamma,\omega}f(\delta, \omega)=f(\delta,q^{n\gamma}\omega)$.

Using this operator we present some other properties of $q$-Prabhakar operators.

\begin{thm}
Let $f \in L_{q}^{1}[a,b]$, ${_{}^{P}}I_{q,a+}^{\alpha ,1-\beta, -\gamma, \omega}f\in AC_q[a,b]$ and $\alpha$, $\beta$, $\gamma$, $\omega \in \mathbb{R}$ with $\alpha>0$, $0<\beta \le 1$. Then
\begin{eqnarray}\label{3.13}
\left({_{}^{P}}I_{q,a+}^{\alpha ,\beta, \gamma, \omega'}{_{}^{P}}D_{q,a+}^{\alpha ,\beta ,\gamma ,\omega } f\right)(x)&=&
f(x)\nonumber\\
&-& g_{\gamma,\omega'}^{\alpha,\beta}(x,a/q)\left({_{}^{P}}I_{q,a+}^{\alpha ,1-\beta, -\gamma, \omega}f\right)(a+).
\end{eqnarray}

\end{thm}

\begin{proof}
Let
\begin{eqnarray}\label{3.14}
\varphi(x):=
\left({_{}^{P}}I_{q,a+}^{\alpha ,\beta, \gamma, \omega'}{_{}^{P}}D_{q,a+}^{\alpha ,\beta ,\gamma ,\omega }f\right)(x).
\end{eqnarray}
Applying the Prabhakar fractional $q$-derivative  $_{q}^{P}D_{x,a+}^{\alpha ,\beta ,\gamma , \omega }$ to the both sides of (\ref{3.14}) and using  Theorem \ref{Theorem3.6}, we obtain
\begin{eqnarray}\label{3.15}
    {_{}^{P}}D_{q,a+}^{\alpha ,\beta ,\gamma ,\omega} \varphi= {_{}^{P}}D_{q,a+}^{\alpha ,\beta ,\gamma ,\omega }f.
\end{eqnarray}

We apply $\Lambda_{q}^{2\gamma,\omega}$ operator to the equality (\ref{3.15}). Then, taking into account $w'=q^{\gamma}\omega$, we get
\begin{eqnarray}\label{3.15new}
    {_{}^{P}}D_{q,a+}^{\alpha ,\beta ,\gamma ,q^{\gamma}\omega'} \varphi= {_{}^{P}}D_{q,a+}^{\alpha ,\beta ,\gamma ,q^{\gamma}\omega'}f.
\end{eqnarray}

From the last equality, we conclude that $f-\varphi$ is an element of the kernel of the Prabhakar fractional $q$-differential operator, i.e.,
\begin{eqnarray*}
    f-\varphi\in \ker(_{}^{P}D_{q,a+}^{\alpha ,\beta ,\gamma ,q^{\gamma}\omega'}).
\end{eqnarray*}

Introducing notation $\psi:=f-\varphi$ and taking into account (\ref{3.11}) and $0<\beta \leq 1$, we have
\begin{eqnarray*}
    \left(_{}^{P}D_{q,a+}^{\alpha ,\beta ,\gamma ,q^{\gamma}\omega' }\psi\right)(x)=0  \Leftrightarrow  D_{q}\left( _{}^{P}I_{q,a+}^{\alpha ,1-\beta ,-\gamma , q^{\gamma}\omega' } \psi\right)\left( x \right)=0.
\end{eqnarray*}

By the standard properties of the $q$-differential operator the last equality means that    $\left(_{}^{P}I_{q,a+}^{\alpha, 1-\beta ,-\gamma ,q^{\gamma}\omega'}\psi\right)(x)$ must be a constant:
\begin{eqnarray}\label{3.16}
    \left(_{}^{P}I_{q,a+}^{\alpha, 1-\beta ,-\gamma ,q^{\gamma}\omega' }\psi\right)(x)=a_0,
\end{eqnarray}
where $a_0$ is an arbitrary constant.

Hence, applying the Prabhakar $q$-fractional differential operator  $_{}^{P}D_{q,a+}^{\alpha, 1-\beta ,-\gamma ,\omega' }$ to the last equality and using Theorem \ref{Theorem3.6}, we find
\begin{eqnarray*}
\psi(x) = {_{}^{P}}D_{q,a+}^{\alpha, 1-\beta ,-\gamma ,\omega'}(a_0).
\end{eqnarray*}

By Definitions \ref{def3.1} and \ref{def3.5} of  the Prabhakar fractional $q$-differential and $q$-integral operators and also   (\ref{2.13}) expansion of the generalized $q$-Prabhakar function, we have
\begin{eqnarray*}
    \psi(x) &=& a_0{{D}_{q}}\int\limits_{a}^{x}{\left( x-qt \right)_{q}^{\beta -1}\sum\limits_{n=0}^{+\infty }{\frac{{{\left( \gamma  \right)}_{n,q}}{{(\omega') }^n}}{{{\Gamma }_q}\left( \alpha n+\beta  \right)}{{\left( x-{q^{\beta }}t \right)}^{\alpha n}_{q}}d_qt}} \\        &=& a_0\sum\limits_{n=0}^{+\infty }{{{\left( \gamma  \right)}_{n,q}}{{(\omega') }^{n}}}{{D}_{q}}\left[ \frac{1}{{{\Gamma }_{q}}\left( \alpha n+\beta  \right)}\int\limits_{a}^{x}{\left( x-qt \right)_{q}^{\alpha n+ \beta -1}}{{d}_{q}}t \right].
\end{eqnarray*}

Hence, using formulas (\ref{2.7}), (\ref{2.6}) and taking (\ref{2.13}) and (\ref{3.3}) into  account,  we find
\begin{eqnarray*}
\psi(x)=a_0 g_{\gamma,\omega'}^{\alpha,\beta}(x,a/q).
\end{eqnarray*}

Since  $\psi=f-\varphi$, we obtain
\begin{eqnarray}\label{3.17}
    f(x)=\varphi(x)+a_0 g_{\gamma,\omega'}^{\alpha,\beta}(x,a/q).
\end{eqnarray}

Hence, applying
the Prabhakar fractional $q$-integral operator  ${}_{}^{P}I_{q,a+}^{\alpha, 1-\beta ,-\gamma ,\omega }$  to the last equality, we find
\begin{eqnarray}\label{3.18}
(_{}^{P}I_{q,a+}^{\alpha, 1-\beta ,-\gamma , \omega }f) (x)= ( _{}^{P}I_{q,a+}^{\alpha, 1-\beta ,-\gamma, \omega }\varphi)(x)+a_0
 {_{}^{P}I_{q,a+}^{\alpha, 1-\beta ,-\gamma , \omega }}g_{\gamma,\omega'}^{\alpha,\beta}(x,a/q).
\end{eqnarray}

Using Proposition \ref{prop3.1} and taking $e_{\alpha,\beta}^{0}(x)=1$ into account it is easy to show that
\begin{eqnarray}\label{3.19new}
    _{}^{P}I_{q,a+}^{\alpha, 1-\beta ,-\gamma , \omega }g_{\gamma,\omega'}^{\alpha,\beta}(x,a/q)=g_{0,\omega'}^{\alpha,1}(x,a/q)=e_{\alpha,\beta}^{0}[\omega'(x-a)^{\alpha}]=1.
\end{eqnarray}
Then from (\ref{3.18}), we obtain
\begin{eqnarray}\label{3.19}
\left({_{}^{P}}I_{q,a+}^{\alpha, 1-\beta ,-\gamma, \omega }f\right)(x)=\left({_{}^{P}}I_{q,a+}^{\alpha, 1-\beta ,-\gamma ,\omega }\varphi\right)(x)+a_0.
\end{eqnarray}

Taking into account the notation  (\ref{3.14}) and using  (\ref{3.4}) and Definition \ref{def3.5}, we have
\begin{eqnarray*}  \left({_{}^{P}}I_{q,a+}^{\alpha, 1-\beta ,-\gamma, \omega }\varphi\right)(x)&=&\left({_{}^{P}}I_{q,a+}^{\alpha, 1-\beta ,-\gamma, \omega }{_{}^{P}}I_{q,a+}^{\alpha ,\beta, \gamma, \omega'}{_{}^{P}}D_{q,a+}^{\alpha ,\beta ,\gamma ,\omega } f\right)(x)\\&=&
\left({_{}^{P}}I_{q,a+}^{\alpha ,1,0,\omega} {_{}^{P}}D_{q,a+}^{\alpha ,\beta ,\gamma , \omega }f\right)(x)\\&=&\left(I_{q,a+}^{1} D_q
 \ {_{}^{P}}I_{q,a+}^{\alpha ,1-\beta, -\gamma, \omega}f\right)(x).
\end{eqnarray*}
Hence,  applying fundamental theorem of $q$-calculus (\cite{Kac}), we find
\begin{eqnarray*}    \left({_{}^{P}}I_{q,a+}^{\alpha, 1-\beta ,-\gamma ,\omega }\varphi\right)(x)=\left({_{}^{P}}I_{q,a+}^{\alpha ,1-\beta, -\gamma, \omega}f\right)(x)-\left({_{}^{P}}I_{q,a+}^{\alpha ,1-\beta, -\gamma, \omega}f\right)(a).
\end{eqnarray*}
 Comparing this with (\ref{3.19}), we conclude that
 \begin{eqnarray*}
     a_0=\left({_{}^{P}}I_{q,a+}^{\alpha ,1-\beta, -\gamma, \omega}f\right)(a).
 \end{eqnarray*}
 Substituting obtained expression of $a_0$ into (\ref{3.17}) and considering the notation (\ref{3.14}), we get (\ref{3.13}).

 The proof of Theorem 3.7 is complete.
\end{proof}

\section{A Cauchy type problem associated with $q$-Prabhakar differential operator}

Let us consider the following  Cauchy-type problem with Prabhakar fractional $q$-differential operator:
\begin{eqnarray}\label{4.1}  \left(_{}^{P}D_{q,a+}^{\alpha ,\beta ,\gamma, \omega } y\right) \left( x \right)=f(x,y),
\end{eqnarray}
\begin{eqnarray}\label{4.2}
\left(_{}^{P}I_{q,a+}^{\alpha ,1-\beta ,-\gamma, \omega } y\right) \left( a+ \right)={\xi}_{0},
\end{eqnarray}
where   $ \alpha, \beta, \gamma, \omega, \xi_0 \in \mathbb{R}$ are such that $ \alpha>0,$ $0<\beta\le 1$, $\xi_0\ne 0$.

We prove the existence and uniqueness of the solution to the problem (\ref{4.1})-(\ref{4.2}).

\begin{thm}\label{Theorem4.1}
 Let   $f(\cdot,\cdot ):[a,b]\times\mathbb{R}  \to \mathbb{R} $ be a function such that $f(\cdot,y(\cdot))\in L_{q}^{1}[a,b]$ for all $y\in L_{q}^{1}[a,b]$.

Then $y$ satisfies the relations (\ref{4.1}) and (\ref{4.2}) if and only if $y$ satisfies the following $q$-Volterra integral equation:
\begin{eqnarray}\label{4.3}
    y(x)={_{}^{P}}I_{q,a+}^{\alpha ,\beta ,\gamma, \omega^{\prime} }f(x,y)+\xi_0g_{\gamma,\omega'}^{\alpha,\beta}(x,a/q).
\end{eqnarray}

\end{thm}

\begin{proof}
First, we prove the necessity. We assume that $y \in L_{q}^{1}[a,b]$ satisfies (\ref{4.1})-(\ref{4.2}). Since $f(x,y)\in L_{q}^{1}[a,b]$, (\ref{4.1}) means that there exists  on $[a,b]$ a Prabhakar fractional $q$-differential ${_{}^{P}}D_{q,a+}^{\alpha ,\beta ,\gamma, \omega }y\in L_{q}^{1}[a,b].$ So we can  apply operator  $_{}^{P}I_{q,a+}^{\alpha ,\beta ,\gamma, \omega^{\prime}}$
to the equation (\ref{4.1}). Then, considering the formula (\ref{3.13}) and  the condition (\ref{4.2}), we get the integral equation (\ref{4.3}).

Now, we prove the sufficiency. Let $y\in L_{q}^{1}[a,b]$ satisfy
the equation
(\ref{4.3}). Applying the operator ${_{}^{P}}D_{q,a+}^{\alpha ,\beta, \gamma, {\omega}}$ to both sides of (\ref{4.3}) and using (\ref{3.12}), we get
\begin{eqnarray}\label{4.4}
\left({_{}^{P}}D_{q,a+}^{\alpha ,\beta, \gamma, \omega}y\right)(x)-f(x,y)=\xi_0{_{}^{P}}D_{q,a+}^{\alpha ,\beta, \gamma, \omega} g_{\gamma,\omega'}^{\alpha,\beta}(x,a/q).
\end{eqnarray}

We show that the right-hand side of (\ref{4.4}) is equal to zero. Using (\ref{3.11}) and  (\ref{3.19new}), we find
\begin{eqnarray*}
{_{}^{P}}D_{q,a+}^{\alpha ,\beta, \gamma, \omega} g_{\gamma,\omega'}^{\alpha,\beta}(x,a/q)=D_q{_{}^{P}}I_{q,a+}^{\alpha ,1-\beta, -\gamma, \omega}g_{\gamma,\omega'}^{\alpha,\beta}(x,a/q) =D_q(1)=0.
\end{eqnarray*}

Now, we show that the relation in (\ref{4.2}) is also held. For this, we apply the operator
$_{}^{P}I_{q,a+}^{\alpha ,1-\beta ,-\gamma, \omega }$
to both sides of (\ref{4.3}). Considering  (\ref{3.4}) and (\ref{3.19new}), we get
\begin{eqnarray} \label{4.5}   \left({_{}^{P}}I_{q,a+}^{\alpha ,1-\beta ,-\gamma, \omega }y\right)(x)-\int\limits_{a}^{x}{f(t,y(t))d_qt}=\xi_0.
\end{eqnarray}
By putting $x=+a$ in (\ref{4.5})  we obtain the relation in (\ref{4.2}).

 Theorem \ref{Theorem4.1} is proved.
\end{proof}

\section{Existence and uniqueness of the solution to the Cauchy type problem}

  In this section, we prove the existence and uniqueness of the solution to the problem (\ref{4.1})-(\ref{4.2}). The result is obtained under the conditions of Theorem \ref{4.1} and  Lemma \ref{lem3.4}.

\begin{thm}
Let $G$ be an open set in $\mathbb{R}$.  Let $f(\cdot,\cdot ):[a,b]\times G \to \mathbb{R} $ be a function such that $f(\cdot,y(\cdot))\in L_q^{1}[a,b]$ for all $y \in G$, and for all $x\in(a,b]$ and for all $y_1,y_2\in G$, it satisfies
\begin{eqnarray}
\label{5.1}
\left|f(x,y_1)-f(x,y_2)\right|\le A \left|y_1-y_2\right|,
\end{eqnarray}
where $A>0$ does not depend on $x\in[a,b]$ and $y_1,y_2\in L_{q}^{1}[a,b]$.

Then there exists a unique solution $y\in L_{q}^{1}[a,b]$ to  the problem (\ref{4.1})-(\ref{4.2}).
\end{thm}

\begin{proof}
According to Theorem \ref{Theorem4.1}, the problem (\ref{4.1})-(\ref{4.2}) is equivalent to the integral equation (\ref{4.3}). So to prove the existence and uniqueness of the solution to the problem (\ref{4.1})-(\ref{4.2}) it is sufficient to show the existence and uniqueness of the solution to the integral equation (\ref{4.3}). To do this we rewrite the integral equation (\ref{4.2}) in the following operator form
\begin{eqnarray}\label{5.2}
    y(x)=(Ty)(x),
\end{eqnarray} where
\begin{eqnarray}\label{5.3}
    (Ty)(x):=y_0(x)+\int\limits_{a}^{x}(x-qt)^{\beta-1}_{q}e_{\alpha,\beta}^{\gamma}\left[ \omega'(x-q^{\beta}t)^{\alpha}_{q}\right] f[t,y(t)]d_qt
\end{eqnarray}
and
\begin{eqnarray*}
    y_0(x):=\xi_0 (x-a)_{q}^{\beta-1}e_{\alpha ,\beta}^{\gamma}\left[ \omega' {(x-q^{\beta-1}a)_{q}^{\alpha}};q \right].
\end{eqnarray*}

First, we prove the existence of a unique solution $y(x)$ in the space $L_{q}^{1}[a,b]$. Our proof is based on the Banach fixed point theorem. We should note that $L_q^{1}[a,b]$ is a complete metric space (\cite{Zaynab}).

Select $h\in(a,b]$ such that
\begin{eqnarray}\label{5.4}
   \delta_1=A(h-qa)_q^{\beta}e_{\alpha, \beta+1}[\omega'(h-q^{\beta+1}a)_{q}^{\alpha};q]<1,
\end{eqnarray}
where $A>0$  is the Lipschitz constant in  (\ref{5.1}). Clearly $y_0\in L_q^{1}[a,h]$. Also, by Lemma \ref{lem3.4} $(Ty)(x)\in L_q^{1}[a,h]$.   Therefore $T$ maps
$L_q^{1}[a,h]$ into itself. Moreover, from (\ref{5.1}),(\ref{5.3}) and Lemma \ref{lem3.4}, for any $y_1$, $y_2\in L_{q}^{1}[a,h]$, we have
\begin{eqnarray*}
    \left\| Ty_1-Ty_2 \right\|_{L_{q}^{1}[a,h]}&\le&  \left\| {{_{q}^{P}}I_{x,a+}^{\alpha ,\beta ,\gamma, \omega }f(x,y_1(x))-{_{q}^{P}}I_{x,a+}^{\alpha ,\beta ,\gamma, \omega }f(x,y_2(x))} \right\|_{L_{q}^{1}[a,h]}\\& \leq&
    (h-qa)_q^{\beta}e_{\alpha, \beta+1}[\omega'(h-q^{\beta+1}a)_{q}^{\alpha};q]\\&\times&\left\| {f(x,y_1(x))-f(x,y_2(x))} \right\|_{L_{q}^{1}[a,h]}
    \\& \leq&
    A(h-qa)_q^{\beta}e_{\alpha, \beta+1}[\omega'(h-q^{\beta+1}a)_{q}^{\alpha};q]\left\| {y_1(x)-y_2(x)} \right\|_{L_{q}^{1}[a,h]}
    \\& \leq&
\delta_1\left\| {y_1(x)-y_2(x)} \right\|_{L_{q}^{1}[a,h]}.
\end{eqnarray*}

Our assumption (\ref{5.4}) allows us to apply the Banach fixed point theorem to obtain a unique solution $y^{*}\in L_{q}^{1}[a,h]$ to equation (\ref{5.2}) on the interval $(a,h]$. According to this theorem $y^{*}$ is obtained as a limit of a convergent sequence $(T^{m}y_0)(x)$:
\begin{eqnarray*}
    \underset{m\to\infty}{lim}\left\| {T^{m} \overline{y}(x)-y^{*}} \right\|_{L_{q}^{1}[a,h]}=0
\end{eqnarray*}
in the space $L_{q}^{1}[a,h]$, where  $\overline{y}(x)$ is an arbitrary function in  $L_{q}^{1}[a,h]$.

Since $\xi_0\ne 0$ and $y_0\in L_{q}^{1}[a,h]$ we can take $y_0(x)$ as  $\overline{y}(x)$:
\begin{eqnarray*}
    \overline{y}(x):=y_0(x).
\end{eqnarray*}

Consequently, the sequence $T^{m}y_0(x)$ is defined by the following recurrence relation
\begin{eqnarray*}
    T^{m}y_0(x)=y_0(x)+\int\limits_{a}^{x}(x-qt)^{\beta-1}_{q}e_{\alpha,\beta}^{\gamma}\left[ \omega'(x-q^{\beta}t)^{\alpha}_{q}\right] f[t,T^{m-1}y_0(t)]d_qt,   m\in \mathbb{N}.
\end{eqnarray*}

If we denote $y_m(x)=(T^{m}y_0)(x)$, then the last relation takes the form
\begin{eqnarray*}
   y_m(x)=y_0(x)+\int\limits_{a}^{x}(x-qt)^{\beta-1}_{q}e_{\alpha,\beta}^{\gamma}\left[ \omega'(x-q^{\beta}t)^{\alpha}_{q}\right] f[t,y_{m-1}(t)]d_qt, \quad  m\in \mathbb{N}.
\end{eqnarray*}

This means that the successive approximation method can be used to find a
unique solution of (\ref{4.1})–(\ref{4.2}). \end{proof}

\section*{Acknowledgements}
\noindent
E.~Karimov and A.~Mamanazarov would like to thank the Ghent Analysis \& PDE centre, Ghent University, Belgium for the support during their research visit. The authors
were supported by the FWO Odysseus 1 grant G.0H94.18N: Analysis and Partial
Differential Equations and by the Methusalem programme of the Ghent University Special Research Fund (BOF) (Grant number 01M01021). M.Ruzhansky was also supported
by EPSRC grant EP/R003025/2.

\end{document}